\documentclass[12pt]{amsart}
\usepackage{amssymb,amsmath}
\usepackage{enumitem,graphicx}

\def\C {{\mathbb C}}
\def\H {{\mathcal H}}

\def\A {{\mathbb A}}
\def\R {\mathbb{R}}

\def\D {{{D}}}
\def\i {{{i}}}
\def\z {{\boldsymbol{z}}}
\def\Re {\mathfrak{Re\,}}

\def \l {\langle}
\def \r {\rangle}
\def \and {{\qquad\text{and}\qquad}}

\newtheorem{proposition}{Proposition}[section]
\newtheorem{theorem}[proposition]{Theorem}
\newtheorem{corollary}[proposition]{Corollary}
\newtheorem{lemma}[proposition]{Lemma}
\theoremstyle{definition}
\newtheorem{definition}[proposition]{Definition}
\newtheorem{remark}[proposition]{Remark}
\newtheorem*{Acknowledgments}{Acknowledgments}
\numberwithin{equation}{section}

\def \au {\rm}
\def \ti {\it}
\def \jou {\rm}
\def \bk {\it}
\def \no#1#2#3 {{\bf #1} (#3), #2.}
\def \eds#1#2#3 {#1, #2, #3.}

\title[Euler-Bernoulli beam with localized structural damping]
{Gevrey regularity for the Euler-Bernoulli beam\\ equation
with localized structural damping}

\author[M. Caggio and F. Dell'Oro]
{Matteo Caggio and Filippo Dell'Oro}

\address{Institute of Mathematics of the Academy of Sciences of the Czech Republic
\newline\indent
\v{z}itn\'a 25, 115 67 Praha 1, Czech Republic}
\email{caggio@math.cas.cz {\rm (M.\ Caggio)}}

\address{Politecnico di Milano - Dipartimento di Matematica
\newline\indent
Via Bonardi 9, 20133 Milano, Italy}
\email{filippo.delloro@polimi.it {\rm (F.\ Dell'Oro)}}


\begin{document}

\begin{abstract}
We study a Euler-Bernoulli beam equation
with localized discontinuous structural damping. 
As our main result, we prove that the associated 
$C_0$-semigroup $(S(t))_{t\geq0}$ is of Gevrey class $\delta>24$ for $t>0$, 
hence immediately differentiable. Moreover, we show that 
$(S(t))_{t\geq0}$ is exponentially stable.
\end{abstract}

\maketitle

\section{Introduction}

\noindent
We analyze the regularity and the stability properties of the $C_0$-semigroup associated to
a Euler-Bernoulli beam equation with localized discontinuous structural damping.
Denoting by $\ell>0$ the length of the beam and fixing $0<\ell_0<\ell$, 
the corresponding model reads
\begin{equation}
\label{E1}
u_{tt}(x,t) + u_{xxxx}(x,t) - (a(x)u_{tx}(x,t))_{x} = 0,
\end{equation}
where $x\in(0,\ell_0)\cup(\ell_0,\ell)$ and $t>0$. 
The unknown $u(x,t)$ represents the vertical deflection of the beam
with respect to its reference configuration, while $a(x)$ is the damping function.
We assume that $a(x)$ has the form
$$
a(x)=\begin{cases}
0\quad\,\,\, {\rm if}\,\,\, x\in(0,\ell_0),\\
1\quad\,\,\, {\rm if}\,\,\, x\in(\ell_0,\ell),
\end{cases}
$$
meaning that the dissipation is localized on the interval
$(\ell_0,\ell)$.
We also prescribe the following transmission conditions at the interface
\begin{equation}
\label{trans}
\begin{cases}
u(\ell_0^-,t)=u(\ell_0^+,t),\qquad \quad\,\,\,\, u_x(\ell_0^-,t)=u_x(\ell_0^+,t),\\
u_{xx}(\ell_0^-,t)=u_{xx}(\ell_0^+,t),
\qquad u_{xxx}(\ell_0^-,t) = u_{xxx}(\ell_0^+,t) - u_{tx}(\ell_0^+,t).
\end{cases}
\end{equation}
Finally, we impose the clamped boundary conditions
\begin{equation}
\label{E2}
u(0,t)=u_x(0,t)=u(\ell,t)=u_x(\ell,t) =0,
\end{equation}
and the initial conditions 
\begin{equation}
\label{E3}
u(x,0)=u_0(x),\qquad u_t(x,0) = U_0(x),
\end{equation}
where $u_0,U_0$ are assigned data. 

The analysis of solutions to one-dimensional and multi-dimensional Euler-Bernoulli 
equations with different damping mechanisms is a classical topic in PDEs.
The reader may consult for instance \cite{BELA,KOM,LAG,LAS1,LAS} and many references therein for 
an overview of the relevant literature. The specific problem \eqref{E1}-\eqref{E3} treated here
has been previously studied in \cite{DENK}, actually in a more general higher-dimensional version,
to which we refer for detailed physical motivations.
In that paper, it was shown that the associated contraction $C_0$-semigroup is exponentially stable 
provided that the damping function is supported near the whole boundary of the spatial domain. 
The methodology of \cite{DENK} consists in proving that the resolvent of the semigroup generator 
along the imaginary axis is bounded, which allows to apply the classical
Gearhart-Pr\"{u}ss theorem~\cite{GER,PRU} and deduce the exponential
stability. The case when the damping functions is supported near an arbitrary small portion of the boundary 
has been recently analyzed in \cite{AHR} under different boundary conditions
(hinged rather than clamped). By means of suitable Carleman estimates, the authors of \cite{AHR}
proved that the resolvent of the semigroup generator 
along the imaginary axis grows at most at an exponential rate. Combined
with a well-known result of Burq~\cite{BURQ}, this leads to a logarithmic (semiuniform) decay of the semigroup. 
Other works related to Euler-Bernoulli 
equations with locally distributed structural damping include \cite{BARRAZA,TEB}.

The aim of present paper is to analyze both the stability and the
regularity properties of the contraction $C_0$-semigroup 
$(S(t))_{t\geq0}$ associated to \eqref{E1}-\eqref{E3}. As remarked in \cite{LIULIU},
such a semigroup is not analytic. Still, we show here that $(S(t))_{t\geq0}$ 
is of Gevrey class $\delta>24$ for $t>0$. In particular, $(S(t))_{t\geq0}$ turns out to be immediately differentiable,
which implies an instantaneous smoothing effect on the initial data. Moreover, the spectrum determined growth condition
(otherwise called the linear stability property) is satisfied, meaning that the growth bound of the semigroup
coincides with the spectral bound of its generator. 
Finally, we prove that $(S(t))_{t\geq0}$ is exponentially stable. 
Note that in our case the damping function $a(x)$ vanishes on $(0,\ell_0)$, and thus 
it is not supported near the whole boundary of the spatial domain $(0,\ell)$. For this reason,
the exponential stability of $(S(t))_{t\geq0}$ cannot be deduced directly from the results of \cite{DENK}. 
The method used here to prove the aforementioned Gevrey regularity of 
$(S(t))_{t\geq0}$ consists in showing that the resolvent of the semigroup generator 
along the imaginary axis decays at a proper polynomial rate, which permits
to reach the desired conclusion by invoking an abstract result of Taylor~\cite{TAYLOR}.
 
In order to gain a better understanding of our results, we shall compare them with 
the analogous ones obtained for the case of the Kelvin-Voigt damping, namely, when the damping
term and the leading elastic term have the same spatial order (four).
As shown in \cite[Theorem~4.1]{LIULIU}, the $C_0$-semigroup associated to
the Euler-Bernoulli beam equation with localized discontinuous Kelvin-Voigt damping is exponentially stable but not 
analytic. Still, it has been recently proved in \cite{SOZZO} that such a semigroup is of Gevrey class $\delta > 8$.
The difference between the two Gevrey orders of regularity (order $8$ for the Kelvin-Voigt damping and order $24$
for the structural damping) reflects the fact 
that the structural damping is weaker than the Kelvin-Voigt one.
Indeed, the Gevrey order is a way to ``measure" the divergence degree of the power series expansion of a non-analytic
function (the larger the order, the ``more divergent" the power expansion). 
From the technical viewpoint, the worse regularizing effect of the structural damping
represents the main difficulty with respect to the analysis 
carried out in \cite{LIULIU,SOZZO}.
More precisely, one needs a proper estimate of the solution 
in the damped region of the spatial domain, which in our case 
translates into an estimate of $u$ in the Sobolev space $H^2(\ell_0,\ell)$.
Such an estimate can be obtained basically for free for the Kelvin-Voigt damping,
but this is not the case for the structural damping treated here. 

The paper is organized as follows. In the next Section \ref{SEC2} we rewrite the problem 
as a coupled PDE system in two unknown variables $v$ and $w$. In the subsequent 
Section~\ref{SEC3} we recast such a system as 
an abstract Cauchy problem and we prove the existence of the contraction $C_0$-semigroup 
$(S(t))_{t\geq0}$.
We also show that the spectrum of the semigroup generator is contained in the open 
left half-plane $\C^{-}$, meaning in particular that the imaginary axis is spectrum free.
In Section \ref{SEC4} we state the main Theorem~\ref{MAINTH} 
concerning the Gevrey regularity of $(S(t))_{t\geq0}$, together with the corollaries
concerning the differentiability, the spectrum determined growth condition and the exponential stability. 
The final Section \ref{SEC5} is devoted to the proof of Theorem \ref{MAINTH}.

\section{Preliminaries}
\label{SEC2}

\subsection{Rewriting the problem} 
Splitting the unknown variable $u$ as
$$
u(x,t)=
\begin{cases}
v(x,t)\quad\,\,\, \,{\rm if}\,\,\, x\in(0,\ell_0)\\
w(x,t)\quad\,\,\, {\rm if}\,\,\, x\in(\ell_0,\ell),
\end{cases}
$$
we rewrite equation \eqref{E1} as the coupled PDE system (in the sequel we shall write $v$ and $w$ 
instead of $v(x,t)$ and $w(x,t)$ to avoid cumbersome notation)
\begin{equation}
\label{E4}
\begin{cases}
v_{tt} + v_{xxxx} = 0\qquad &{\rm in}\,\,\, (0,\ell_0)\times (0,\infty),\\
w_{tt} + w_{xxxx} - w_{txx} = 0\qquad &{\rm in}\,\,\, (\ell_0,\ell) \times (0,\infty).
\end{cases}
\end{equation}
The transmission conditions \eqref{trans} take the form
\begin{equation}
\label{E5}
\begin{cases}
v(\ell_0,t)=w(\ell_0,t),\qquad \quad\,\,\,\, v_x(\ell_0,t)=w_x(\ell_0,t),\\
v_{xx}(\ell_0,t)=w_{xx}(\ell_0,t),
\qquad v_{xxx}(\ell_0,t) = w_{xxx}(\ell_0,t) - w_{tx}(\ell_0,t),
\end{cases}
\end{equation}
while the boundary conditions \eqref{E2} read
\begin{equation}
\label{E2rew}
v(0,t)=v_x(0,t)=w(\ell,t)=w_x(\ell,t) =0.
\end{equation}
Finally, the initial conditions \eqref{E3} can be written in terms of $v$ and $w$ as
\begin{equation}
\label{E3rew}
\begin{cases}
v(x,0)=v_0(x),\qquad\,\,\, v_t(x,0) = V_0(x),\\
w(x,0)=w_0(x),\qquad w_t(x,0) = W_0(x),
\end{cases}
\end{equation}
where the functions $v_0,V_0,w_0,W_0$ are related to $u_0$ and $U_0$ in an obvious way.

\subsection{Functional setting} 
We consider the (complex) Hilbert spaces
\begin{align*}
H^2_l(0,\ell_0) &=\{v \in H^2(0,\ell_0) : v(0)=v_{x}(0)=0 \},\\
H^2_r(\ell_0,\ell) &=\{w \in H^2(\ell_0,\ell) : w(\ell)=w_x(\ell)=0 \},
\end{align*}
equipped with the norms 
$$\|v\|_{H^2_l(0,\ell_0)}=\|v_{xx}\|_{L^2(0,\ell_0)}\qquad\text{and}\qquad
\|w\|_{H^2_r(\ell_0,\ell)}=\|w_{xx}\|_{L^2(\ell_0,\ell)}.
$$
Introducing the variable $\z=(v,V,w,W)$, the 
state space associated to problem \eqref{E4}-\eqref{E3rew} is the (complex) Hilbert space
$$
\H =
\left\{ \z\in 
H^2_l(0,\ell_0) \times L^2(0,\ell_0)\times  H^2_r(\ell_0,\ell)  \times L^2(\ell_0,\ell) \,\, \Big|\,\,
\begin{matrix}
v(\ell_0)=w(\ell_0)\\
v_{x}(\ell_0)=w_{x}(\ell_0)
\end{matrix}
\right\},
$$
equipped with the norm
$$
\|\z\|_{\H}^2 = \|v_{xx}\|_{L^2(0,\ell_0)}^2 + \|V\|_{L^2(0,\ell_0)}^2 +
\|w_{xx}\|_{L^2(\ell_0,\ell)}^2 + \|W\|_{L^2(\ell_0,\ell)}^2.
$$

\subsection{Two technical tools}
The following functional inequalities 
(see e.g. \cite[Ch.\ 5]{ADAMS}) will be crucial for our purposes.

\begin{lemma}[Interpolation]
\label{interpol}
Let $I\subset \R$ be an interval and let $0\leq j< k < m$.
Then there exists a constant $c>0$ such that
$$
\|f\|_{H^k(I)} \leq c \|f\|_{H^j(I)}^{1-\theta} \|f\|_{H^m(I)}^{\theta}\qquad \forall f \in H^m(I),
$$
where $\theta= (k-j)/(m-j)$.
\end{lemma}

\begin{lemma}[Gagliardo-Nirenberg] 
\label{gagl}
Let $I\subset \R$ be a bounded interval. 
Then there exists a constant $c>0$ such that
$$
\|f\|_{L^\infty (I)} \leq c \|f\|_{L^2(I)}^{1/2} \|f\|_{H^1(I)}^{1/2}\qquad \forall f \in H^1(I).
$$
\end{lemma}

\subsection{General agreements}
Throughout the paper, the H\"older and Poincar\'e inequalities 
will be tacitly used several times.
Given $X,Y\geq0$, we will also employ the Young inequality (often without explicit mention) 
\begin{equation}
\label{young}
X^{\,2/p}\,Y^{\,2/q} \leq X^2 + Y^2 \qquad\text{where}\qquad \frac{1}{p}+\frac{1}{q}=1.
\end{equation}
Finally, we denote by $L(\H)$ the space of bounded linear operators on $\H$.

\section{The Semigroup and Its Generator}
\label{SEC3} 

\noindent
Setting $\z(t)=(v(t),V(t),w(t),W(t))^{\rm T}\in\H$,
we may rewrite problem \eqref{E4}-\eqref{E3rew} as 
$$
\begin{cases}
\z'(t)= \A \z(t),\quad t>0,\\ 
\z(0)=\z_0,
\end{cases}
$$
where $\z_0=(v_0,V_0,w_0,W_0)^{\rm T}\in\H$ and the 
operator $\A:\D(\A) \subset \H \to \H$ is defined as
$$
\A\left(\begin{matrix}
v\\\noalign{\vskip1mm}
V\\
w\\\noalign{\vskip1.3mm}
W
\end{matrix}
\right)
=\left(
\begin{matrix}
V\\
-v_{xxxx}\\\noalign{\vskip1.5mm}
W\\
-w_{xxxx} + W_{xx}
\end{matrix}
\right)
$$
with domain
$$
\D(\A) =
\left\{\z\in\H \left|\,\,
\begin{matrix}
v \in H^4(0,\ell_0)\\
V \in H^2_l(0,\ell_0)\\
w \in H^4(\ell_0,\ell)\\
W \in H^2_r(\ell_0,\ell)\\
V(\ell_0)=W(\ell_0)\\
V_x(\ell_0)=W_x(\ell_0)\\
v_{xx}(\ell_0)=w_{xx}(\ell_0)\\
v_{xxx}(\ell_0) = w_{xxx}(\ell_0) - W_{x}(\ell_0)\\
\end{matrix}\right.
\right\}.
$$
Denoting by $\l \cdot, \cdot \r_{\H}$ the inner product associated to $\|\cdot\|_{\H}$, 
a direct calculation yields
\begin{equation}
\label{DISSIP}
\Re \l \A \z , \z \r_{\H} = -\|W_{x}\|_{L^2(\ell_0,\ell)}^2\leq0,\quad\,\, \forall \z \in \D(\A).
\end{equation}
In particular, this tells that $\A$ is a dissipative operator. Further properties of $\A$
are contained in the following result.

\begin{lemma}
\label{propA}
The operator $\A$ is bijective and the inverse operator $\A^{-1}$ is compact.
\end{lemma}

\begin{proof}
First, we show that $\A$ is bijective. To this end, 
for every $\hat \z = (\hat v, \hat V, \hat w, \hat W)^{\rm T}\in\H$, 
we prove that the equation $\A\z=\hat \z$ 
has a unique solution $\z=(v,V,w,W)^{\rm T}\in D(\A)$. 
Writing the equation in components, we obtain the system
\begin{align}
\label{E1s}
V = \hat v &\qquad\,\, {\rm in}\,\,\,H^2_l(0,\ell_0),\\
\label{E2s}
-v_{xxxx} = \hat V &\qquad\,\, {\rm in}\,\,\,L^2(0,\ell_0),\\\noalign{\vskip0.5mm}
\label{E3s}
W = \hat w &\qquad\,\, {\rm in}\,\,\,H^2_r(\ell_0,\ell),\\
\label{E4s}
- w_{xxxx} + W_{xx} = \hat W &\qquad\,\, {\rm in}\,\,\,L^2(\ell_0,\ell).
\end{align} 
From \eqref{E1s} and \eqref{E3s} we immediately learn that
$$
V = \hat v \in H^2_l(0,\ell_0) \and W = \hat w \in H^2_r(\ell_0,\ell).
$$
Since $\hat v(\ell_0)=\hat w(\ell_0)$ and $\hat v_x(\ell_0)=\hat w_x(\ell_0)$, we also have
the equalities
$V(\ell_0)=W(\ell_0)$ and $V_x(\ell_0)=W_x(\ell_0)$. 
For $x \in (0,\ell_0)$, the general solution of 
\eqref{E2s} with the boundary conditions $v(0)=v_x(0)=0$ is given by
$$
v(x)= a_1 x^3 + a_2 x^2 + p(x),
$$
where $a_1,a_2\in\C$ and
$$
p(x) = - \int_0^x \int_0^y \int_0^r \int_0^s \hat V(\tau)\, d\tau ds dr dy.
$$ 
Similarly, for $x \in (\ell_0,\ell)$, the general solution of \eqref{E4s} with $W = \hat w$ and
the boundary conditions $w(\ell)=w_x(\ell)=0$ is given by
$$
w(x)= a_3 (x-\ell)^3 + a_4 (x-\ell)^2 + q(x),
$$ 
where $a_3,a_4\in\C$ and
$$
q(x) = \int_\ell^x \int_\ell^y \int_\ell^r \int_\ell^s 
[\hat w_{xx}(\tau) - \hat W(\tau)]\, d\tau ds dr dy.
$$
Note that $v \in H^4(0,\ell_0)\cap H^2_l(0,\ell_0)$ 
and $w \in H^4(\ell_0,\ell)\cap H^2_r(\ell_0,\ell)$. 
In addition, there exists exactly
one choice of the constants $a_1,a_2,a_3,a_4$ such that 
\begin{equation}
\label{sysSSS}
\begin{cases}
v(\ell_0)=w(\ell_0),\\
v_{x}(\ell_0)=w_{x}(\ell_0),\\
v_{xx}(\ell_0)=w_{xx}(\ell_0),\\
v_{xxx}(\ell_0) = w_{xxx}(\ell_0) - W_{x}(\ell_0).
\end{cases}
\end{equation}
Indeed, introducing the matrix
$$
{\bf M}=
\begin{pmatrix}
\ell_0^3 & \ell_0^2 & -(\ell_0-\ell)^3 & -(\ell_0-\ell)^2\\[0.2em]
3 \ell_0^2 & 2\ell_0 & -3(\ell_0-\ell)^2 & - 2(\ell_0-\ell)\\[0.2em]
6\ell_0 & 2 & - 6(\ell_0-\ell) & -2\\[0.2em]
6 & 0 & -6 & 0
\end{pmatrix}
$$
and the vectors
\begin{align*}
&\boldsymbol{a}= (a_1, a_2, a_3,a_4)^{\rm T},\\
&\boldsymbol{b}= (q(\ell_0)-p(\ell_0), q_x(\ell_0)-p_x(\ell_0),
q_{xx}(\ell_0)-p_{xx}(\ell_0),q_{xxx}(\ell_0)-p_{xxx}(\ell_0) - \hat w_x(\ell_0))^{\rm T},
\end{align*}
we rewrite \eqref{sysSSS} as
${\bf M} \hspace{0.1mm}\boldsymbol{a}=\boldsymbol{b}$.
Direct calculations show that ${\rm Det}({\bf M})= 12 \ell^4 \neq 0$, meaning that
there is exactly one solution $\boldsymbol{a}\in \C^4$.
We conclude that $\z=(v,V,w,W)^{\rm T}\in D(\A)$ is the desired unique solution to the equation $\A\z=\hat \z$.

At this point, we multiply in $L^2(0,\ell_0)$ equation \eqref{E1s}
by $V$ and equation \eqref{E2s}
by $v$. Similarly, we multiply in $L^2(\ell_0,\ell)$ equation \eqref{E3s}
by $W$ and equation \eqref{E4s}
by $w$. Adding the resulting identities and after a straightforward calculation, we find
the bound
$$
\|\z\|_\H \leq c\|\hat \z\|_\H,
$$
for some structural constant $c>0$.
The latter ensures that the inverse operator $\A^{-1}$ is bounded, hence $\A$ is closed.
Finally, it is not difficult to check that $D(\A)$ is compactly embedded into $\H$, so that
the operator $\A^{-1}$ is compact.
\end{proof}

In the light of Lemma \ref{propA}, we infer that
$0$ belongs to the resolvent set $\rho(\A)$ of $\A$. Being $\rho(\A)$ an open set, 
this implies that the operator $\lambda-\A$ is surjective
for $\lambda>0$ small enough. In turn, 
$\A$ is densely defined \cite[Corollary II.3.20]{ENG} and, due to
the Lumer-Phillips theorem, it generates of a contraction $C_0$-semigroup 
$$
(S(t))_{t\geq0}:\mathcal{H}\to\mathcal{H}.
$$
We end the section by proving a spectral property of $\A$ that will be crucial
in the sequel.

\begin{lemma} 
\label{operA}
The spectrum $\sigma(\A)$ of $\A$ is contained in the open left half-plane $\C^{-}$. In particular,
the imaginary axis $i\R$ is spectrum free.
\end{lemma}

\begin{proof}
Being $\A$ the infinitesimal generator 
of a contraction $C_0$-semigroup, it follows from the Hille-Yosida theorem that $\sigma(\A)$ is
contained in the closed left half-plane. Since 
we know from Lemma \ref{propA} that $\A^{-1}$ is a compact operator, 
$\sigma(\A)$ consists entirely of isolated eigenvalues (see e.g.~\cite[Theorem~6.29]{Kat}). 
We are left to show that there are no purely imaginary eigenvalues of $\A$. 
To this end, assume that $\z=(v,V,w,W)^{\rm T}\in D(\A)$ satisfies 
\begin{equation}
\label{SPEEQ}
i\lambda \z-\A\z=0
\end{equation}
for some $\lambda\in\R$ with $\lambda\neq0$ 
(recall that $0\notin\sigma(\A)$). Then, we get the system 
\begin{align}
\label{SPE1}
i\lambda v - V =0,\\
\label{SPE2}
i\lambda V + v_{xxxx} = 0,\\
\label{SPE3}
i\lambda w -W =0,\\
\label{SPE4}
i\lambda W + w_{xxxx} - W_{xx} = 0.
\end{align}
Taking the inner product in $\H$ of \eqref{SPEEQ} with $\z$
and invoking \eqref{DISSIP}, we see at once that $W_x=0$, yielding in turn $W=0$. Thus,
it follows from $\eqref{SPE3}$ that $w=0$ as well. Next, substituting \eqref{SPE1} into \eqref{SPE2},
we infer that $v$ fulfills 
$$
v_{xxxx}=\lambda^2 v.
$$
For $x \in (0,\ell_0)$, the general solution of 
the equation above with the boundary conditions $v(0)=v_x(0)=0$ is given by
$$
v(x)= 
a_1 {\rm e}^{i|\lambda|^{1/2} x}
+a_2 {\rm e}^{-i|\lambda|^{1/2} x}
+\alpha(a_1,a_2) {\rm e}^{|\lambda|^{1/2} x} + \beta(a_1,a_2) {\rm e}^{-|\lambda|^{1/2}x} 
$$
where $a_1,a_2\in\C$ and
$$
\alpha(a_1,a_2) = -\frac12 [a_1+a_2 + i (a_1-a_2)],\quad\,\,\,
\beta(a_1,a_2) = -\frac12 [a_1+a_2 - i (a_1-a_2)].
$$
Since $w(\ell_0)=w_{x}(\ell_0)=w_{xx}(\ell_0)=w_{xxx}(\ell_0)-W_x(\ell_0)=0$, 
and recalling that $\z\in D(\A)$,
the function $v$ is subjected to the further constraints 
$v(\ell_0)=v_x(\ell_0)=v_{xx}(\ell_0)=v_{xxx}(\ell_0)=0$. It is straightforward 
to check that these conditions
force $v=0$. Finally, equation
\eqref{SPE1} ensures that $V=0$.
\end{proof}

\section{The Gevrey Class}
\label{SEC4}

\noindent
In this section, we state the main result of the article concerning the Gevrey
class of the $C_0$-semigroup $(S(t))_{t\geq0}$ generated by the operator $\A$.
Let us recall the definition. 

\begin{definition}
Let $\delta>1$ and $t_0\geq0$ be given. A $C_0$-semigroup $(S(t))_{t\geq0}$  
is said to be of {\it Gevrey class $\delta$} for $t>t_0$ if
$(S(t))_{t\geq0}$ is infinitely differentiable for $t>t_0$ and
for all compact $K\subset (t_0,\infty)$ and all $\theta>0$ 
there is a constant $C = C(K,\theta)>0$ such that
$$
\|S^{(n)}(t)\|_{L(\H)}\leq C \theta^n (n!)^\delta
$$
for all $ t \in K$ and $n=0,1,2,\ldots$. Here,
$S^{(n)}(t)$ denotes the $n$-th derivative of $(S(t))_{t\geq0}$.
\end{definition}

\begin{remark}
Recall that if $(S(t))_{t\geq0}$ is infinitely differentiable for $t>t_0$ 
then $\|S^{(n)}(t)\|_{L(\H)}$ is continuous 
for all $t>t_0$ and $n=0,1,2,\ldots$ (see e.g.\ \cite[Lemma 4.2]{PAZ}).
\end{remark}

\begin{remark}
The regularity properties of Gevrey semigroups are somewhat ``between"
that of differentiable semigroups and analytic semigroups. Indeed,
the bounds on the derivatives of $(S(t))_{t\geq0}$ are 
stronger than the ones corresponding to differentiability
but weaker than the ones corresponding to
analyticity (see e.g.\ \cite{CT,PAZ,TAYLOR}).
\end{remark}

Our main result reads as follows.

\begin{theorem}
\label{MAINTH}
The $C_0$-semigroup $(S(t))_{t\geq0}$
is of Gevrey class $\delta>24$ for $t>0$. 
\end{theorem}

In particular, $(S(t))_{t\geq0}$ turns out to be immediately differentiable.
As a consequence,
the spectrum determined growth condition is satisfied (see e.g.\ \cite[p.\,281]{ENG}), namely,
the spectral bound of $\A$
$$\sigma_*=\sup \{ \Re (\zeta) : \zeta \in \sigma(\A) \}$$
coincides with the growth bound of $(S(t))_{t\geq0}$
$$
\omega_*= \inf\{\omega\in\R:\,\,\|S(t)\|_{L(\H)}\leq
Me^{\omega t}\,\text{ for some }M= M(\omega)\geq 1\}.
$$
Moreover, being $(S(t))_{t\geq0}$ immediately 
differentiable (hence eventually norm-continuos), it follows from the inclusion 
$\sigma(\A)\subset\C^{-}$ ensured by Lemma~\ref{operA} that $(S(t))_{t\geq0}$
is exponentially stable, i.e.\ $\omega_*<0$ (see e.g.\ the proof of \cite[Prop.\,2.7]{AB}).
We summarize this discussion in the next result.

\begin{corollary}
The $C_0$-semigroup $(S(t))_{t\geq0}$ is immediately differentiable
and the spectrum determined 
growth condition is satisfied.
Moreover, $(S(t))_{t\geq0}$ is exponentially stable. 
\end{corollary}

\section{Proof of Theorem \ref{MAINTH}}
\label{SEC5}

\noindent
The argument is based on the following well-known abstract criterion, 
whose proof can be found in \cite[Ch.\,5]{TAYLOR}.

\begin{lemma}[Taylor]
\label{absgev}
Assume that $\sigma(\A) \cap i \R = \emptyset$. If there exists $\mu\in(0,1)$ such that
\begin{equation}
\label{limsup}
\limsup_{|\lambda|\to\infty} |\lambda|^{\mu}\|(i\lambda - \A)^{-1}\|_{L(\H)}<\infty,
\end{equation}
then $(S(t))_{t\geq0}$ is of Gevrey class $\delta> 1/\mu$ for $t>0$.
\end{lemma}

Since Lemma \ref{operA} ensures that $\sigma(\A) \cap \i \R = \emptyset$,
we only need to show the validity of condition \eqref{limsup} with $\mu=1/24$. To this end,
for every $\lambda \in \R$ and
$\hat \z=(\hat v,\hat V,\hat w,\hat W)^{\rm T}\in \H$,
we consider the resolvent equation
$$
i\lambda \z - \A \z = \hat \z,
$$
and we denote by $\z = (v,V,w,W)^{\rm T}\in \D(\A)$ its unique solution. 
Taking the inner product in $\H$ with $\z$
and exploiting \eqref{DISSIP}, we immediately get 
$$
\Re \l i\lambda \z - \A \z, \z \r_\H = \|W_{x}\|_{L^2(\ell_0,\ell)}^2=
\Re \l \hat \z, \z \r_\H.
$$
The latter yields
\begin{equation}
\label{DISS}
\|W_{x}\|_{L^2(\ell_0,\ell)}^2 \leq \|\z\|_{\H}\|\hat \z \|_{\H}.
\end{equation}
Next, writing  the resolvent equation componentwise, we obtain the system
\begin{align}
\label{R1s}
i\lambda v - V = \hat v &\qquad\,\, {\rm in}\,\,\,H^2_l(0,\ell_0),\\
\label{R2s}
i\lambda V + v_{xxxx} = \hat V &\qquad\,\, {\rm in}\,\,\,L^2(0,\ell_0),\\\noalign{\vskip0.5mm}
\label{R3s}
i\lambda w - W = \hat w &\qquad\,\, {\rm in}\,\,\,H^2_r(\ell_0,\ell),\\
\label{R4s}
i\lambda W + w_{xxxx} - W_{xx} = \hat W &\qquad\,\, {\rm in}\,\,\,L^2(\ell_0,\ell).
\end{align}
In the sequel, we always denote by $c>0$ a generic constant
depending only on the structural quantities
of the problem (hence independent of $\lambda$),
whose value might change from line to line, or even within the same line.

\begin{lemma}
\label{wW}
For every $|\lambda|\geq1$ the inequality
$$
\|w_{xx}\|_{L^2(\ell_0,\ell)}^2 + \|W\|_{L^2(\ell_0,\ell)}^2 
\leq \frac{c}{|\lambda|^{2/3}}\big[\|\z\|_{\H}^2 + \|\hat \z \|_{\H}^2\big]
$$
holds for some structural constant $c>0$ independent of $\lambda$.
\end{lemma}

\begin{proof}
Exploiting \eqref{R3s} and \eqref{DISS}, it is readily seen that 
\begin{align}
\label{wxest}
\|w_x\|_{L^2(\ell_0,\ell)}\leq \frac{c}{|\lambda|}\big[\|W_{x}\|_{L^2(\ell_0,\ell)}+ 
\|\hat w_{x}\|_{L^2(\ell_0,\ell)}\big]
&\leq\frac{c}{|\lambda|}\big[ \|\z\|_{\H}^{1/2}\|\hat \z \|_{\H}^{1/2}
+\|\hat \z \|_{\H}\big]\\\nonumber
&\leq \frac{c}{|\lambda|}\big[ \|\z\|_{\H} + \|\hat \z \|_{\H}\big].
\end{align}
Moreover, it follows from \eqref{R4s} and \eqref{R3s} that for all $|\lambda|\geq1$
\begin{align}
\label{wxxxxest}
\|w_{xxxx}\|_{L^2(\ell_0,\ell)}&\leq c|\lambda|\|W\|_{L^2(\ell_0,\ell)}
+c\|W_{xx}\|_{L^2(\ell_0,\ell)}+
c\|\hat W\|_{L^2(\ell_0,\ell)}\\\noalign{\vskip0.7mm}\nonumber
&\leq c|\lambda|\|W\|_{L^2(\ell_0,\ell)}+ c|\lambda|\|w_{xx}\|_{L^2(\ell_0,\ell)}
+c \|\hat \z \|_{\H}\\\noalign{\vskip0.7mm}\nonumber
&\leq  c|\lambda|\big[\|\z\|_{\H} + \|\hat \z \|_{\H}\big].
\end{align}
Using \eqref{wxest}-\eqref{wxxxxest} and interpolating between $H^1(\ell_0,\ell)$ and
$H^4(\ell_0,\ell)$ (cf.\ Lemma \ref{interpol}), we find the bound
\begin{align*}
\|w_{xx}\|^2_{L^2(\ell_0,\ell)} &
\leq c \|w\|^{4/3}_{H^1(\ell_0,\ell)}\|w\|^{2/3}_{H^4(\ell_0,\ell)}\\\noalign{\vskip0.7mm}
&\leq c \|w_x\|^2_{L^2(\ell_0,\ell)} 
+ c \|w_x\|^{4/3}_{L^2(\ell_0,\ell)}\|w_{xxxx}\|^{2/3}_{L^2(\ell_0,\ell)}\\
&\leq \frac{c}{|\lambda|^{2/3}}\big[\|\z\|_{\H}^2 + \|\z\|_{\H}^{4/3}\|\hat \z\|_{\H}^{2/3}
+ \|\z\|_{\H}^{2/3}\|\hat \z \|_{\H}^{4/3} + \|\hat \z \|_{\H}^2\big]
\end{align*}
for all $|\lambda|\geq1$. Invoking the Young inequality \eqref{young}, 
we infer that 
$$
\|\z\|_{\H}^2 + \|\z\|_{\H}^{4/3}\|\hat \z\|_{\H}^{2/3}
+ \|\z\|_{\H}^{2/3}\|\hat \z \|_{\H}^{4/3} + \|\hat \z \|_{\H}^2 
\leq c \|\z\|_{\H}^2+\|\hat \z\|_{\H}^2,
$$
yielding the desired estimate for $\|w_{xx}\|_{L^2(\ell_0,\ell)}$. Next,
owing to \eqref{wxest}-\eqref{wxxxxest} and interpolating again between $H^1(\ell_0,\ell)$ and
$H^4(\ell_0,\ell)$, we obtain
\begin{align*}
\|w_{xxx}\|^2_{L^2(\ell_0,\ell)} &\leq c \|w\|^{2/3}_{H^1(\ell_0,\ell)}\|w\|^{4/3}_{H^4(\ell_0,\ell)}\\
&\leq c \|w_x\|^2_{L^2(\ell_0,\ell)} 
+ c \|w_x\|^{2/3}_{L^2(\ell_0,\ell)}\|w_{xxxx}\|^{4/3}_{L^2(\ell_0,\ell)}\\
&\leq c |\lambda|^{2/3}\big[\|\z\|_{\H}^2 + \|\z\|_{\H}^{4/3}\|\hat \z\|_{\H}^{2/3}
+ \|\z\|_{\H}^{2/3}\|\hat \z \|_{\H}^{4/3} + \|\hat \z \|_{\H}^2\big]\\\noalign{\vskip0.7mm}
&\leq c |\lambda|^{2/3}\big[\|\z\|_{\H}^2 + \|\hat \z \|_{\H}^2\big]
\end{align*}
for all $|\lambda|\geq1$. Defining the function
$$
\omega(x) =  W_x(x)- w_{xxx}(x) +\int_{\ell_0}^x \hat W (y) dy, 
$$
it follows from the estimate above and 
\eqref{DISS} that 
\begin{align*}
\|\omega\|_{L^2(\ell_0,\ell)}&\leq c\big[
\|W_{x}\|_{L^2(\ell_0,\ell)}+\|w_{xxx}\|_{L^2(\ell_0,\ell)}
 + \|\hat W\|_{L^2(\ell_0,\ell)} \big]\\
&\leq c |\lambda|^{1/3} \big[ \|\z\|_{\H} + \|\z\|_{\H}^{1/2}\|\hat \z \|_{\H}^{1/2}
+\|\hat \z \|_{\H}\big]\\\nonumber
&\leq c |\lambda|^{1/3} \big[ \|\z\|_{\H} + \|\hat \z \|_{\H}\big].
\end{align*}
Therefore, exploiting \eqref{R4s} and interpolating between $L^2(\ell_0,\ell)$ and
$H^2(\ell_0,\ell)$, we have
\begin{align*}
\|W\|_{L^2(\ell_0,\ell)}^2 = \frac{1}{|\lambda|^2}\|\omega_x\|_{L^2(\ell_0,\ell)}^2 &\leq 
\frac{c}{|\lambda|^2}\|\omega\|_{L^2(\ell_0,\ell)}\|\omega\|_{H^2(\ell_0,\ell)}\\
&\leq \frac{c}{|\lambda|^2}\|\omega\|_{L^2(\ell_0,\ell)}^2
+\frac{c}{|\lambda|^2}\|\omega\|_{L^2(\ell_0,\ell)}\|\omega_{xx}\|_{L^2(\ell_0,\ell)}\\
&=\frac{c}{|\lambda|^2}\|\omega\|_{L^2(\ell_0,\ell)}^2
+\frac{c}{|\lambda|}\|\omega\|_{L^2(\ell_0,\ell)}\|W_x\|_{L^2(\ell_0,\ell)}\\
&\leq\frac{c}{|\lambda|^{2/3}}\big[\|\z\|_{\H}^{2} +\|\z\|_{\H}^{3/2}\|\hat \z \|_{\H}^{1/2} + 
\|\z\|_{\H}^{1/2}\|\hat \z \|_{\H}^{3/2}+\|\hat \z\|_{\H}^{2}\big]
\end{align*}
for all $|\lambda|\geq1$. Due to the Young inequality \eqref{young}, it is immediate to check that
$$
\|\z\|_{\H}^{2} +\|\z\|_{\H}^{3/2}\|\hat \z \|_{\H}^{1/2} + 
\|\z\|_{\H}^{1/2}\|\hat \z \|_{\H}^{3/2}+\|\hat \z\|_{\H}^{2} 
\leq c \|\z\|_{\H}^2+\|\hat \z\|_{\H}^2,
$$
and the proof is finished.
\end{proof}

\begin{lemma}
\label{POINT1}
For every $|\lambda|\geq1$ the inequalities
\begin{align}
\label{P1}
|W(\ell_0)|^2 &\leq \frac{c}{|\lambda|^{1/3}}\big[\|\z\|_{\H}^2 + \|\hat \z \|_{\H}^2\big]\\\noalign{\vskip1mm}
\label{P2}
|W_x(\ell_0)|^2  &\leq c|\lambda|^{2/3}\big[\|\z\|_{\H}^2 + \|\hat \z \|_{\H}^2\big]\\\noalign{\vskip1.5mm}
\label{P3}
|w_{xx}(\ell_0)|^2 &\leq \frac{c}{|\lambda|^{1/6}}\big[\|\z\|_{\H}^2 + \|\hat \z \|_{\H}^2\big]\\\noalign{\vskip1mm}
\label{P4}
|w_{xxx}(\ell_0)|^2  &\leq c|\lambda|^{5/6}\big[\|\z\|_{\H}^2 + \|\hat \z \|_{\H}^2\big]
\end{align}
hold for some structural constant $c>0$ independent of $\lambda$.
\end{lemma}

\begin{proof}
In what follows, we will tacitly employ several times  
the Gagliardo-Nirenberg inequality (see Lemma \ref{gagl}). 
It is also understood that we work
with $|\lambda|\geq1$.

First, in the light of \eqref{DISS} and Lemma \ref{wW}, we have
\begin{align*}
|W(\ell_0)|^2 \leq \|W\|_{L^\infty(\ell_0,\ell)}^2
&\leq c \|W\|_{L^2(\ell_0,\ell)}\|W_x\|_{L^2(\ell_0,\ell)}\\\noalign{\vskip0.8mm}
&\leq \frac{c}{|\lambda|^{1/3}}\big[\|\z\|_{\H}^{3/2}\|\hat \z \|_{\H}^{1/2} 
+ \|\z\|_{\H}^{1/2}\|\hat \z \|_{\H}^{3/2}\big] \\
&\leq \frac{c}{|\lambda|^{1/3}}\big[\|\z\|_{\H}^2 + \|\hat \z \|_{\H}^2\big].
\end{align*}
The latter estimate is nothing but \eqref{P1}.
Next, an exploitation of \eqref{R3s} together with Lemma \ref{wW} yield
\begin{equation}
\label{WXXest}
\|W_{xx}\|_{L^2(\ell_0,\ell)} \leq c |\lambda| \|w_{xx}\|_{L^2(\ell_0,\ell)} 
+ c\|\hat w_{xx}\|_{L^2(\ell_0,\ell)}
\leq c|\lambda|^{2/3}\big[\|\z\|_{\H} + \|\hat \z \|_{\H}\big].
\end{equation}
Therefore, invoking \eqref{DISS}, we infer that
\begin{align*}
|W_x(\ell_0)|^2 \leq \|W_x\|_{L^\infty(\ell_0,\ell)}^2
&\leq c \|W_x\|_{L^2(\ell_0,\ell)}\|W_{xx}\|_{L^2(\ell_0,\ell)}\\
&\leq c|\lambda|^{2/3}\big[\|\z\|_{\H}^{3/2}\|\hat \z \|_{\H}^{1/2} 
+ \|\z\|_{\H}^{1/2}\|\hat \z \|_{\H}^{3/2}\big]\\
&\leq c|\lambda|^{2/3}\big[\|\z\|_{\H}^2 + \|\hat \z \|_{\H}^2\big],
\end{align*}
which is exactly \eqref{P2}. We now observe that,
appealing to \eqref{R4s} together with \eqref{WXXest} and Lemma \ref{wW}, 
one can write 
\begin{align}
\label{wxxxxest2}
\|w_{xxxx}\|_{L^2(\ell_0,\ell)} &\leq c|\lambda|\|W\|_{L^2(\ell_0,\ell)}
+c\|W_{xx}\|_{L^2(\ell_0,\ell)} +\|\hat W \|_{L^2(\ell_0,\ell)}\\\nonumber
&\leq c|\lambda|^{2/3}\big[\|\z\|_{\H} + \|\hat \z \|_{\H}\big].
\end{align}
As a consequence, interpolating between $H^2(\ell_0,\ell)$ and
$H^4(\ell_0,\ell)$ (cf.\ Lemma \ref{interpol}), we get
\begin{align*}
|w_{xx}(\ell_0)|^2 \leq \|w_{xx}\|_{L^\infty(\ell_0,\ell)}^2
&\leq c \|w_{xx}\|_{L^2(\ell_0,\ell)}^2 + c \|w_{xx}\|_{L^2(\ell_0,\ell)}
\|w_{xxx}\|_{L^2(\ell_0,\ell)}\\
&\leq c \|w_{xx}\|_{L^2(\ell_0,\ell)}^2 + c \|w_{xx}\|_{L^2(\ell_0,\ell)}^{3/2}
\|w_{xxxx}\|_{L^2(\ell_0,\ell)}^{1/2}.
\end{align*}
Due to Lemma \ref{wW} and \eqref{wxxxxest2}, the right-hand side above is less
than or equal to 
$$
\frac{c}{|\lambda|^{1/6}}\big[\|\z\|_{\H}^2+ \|\z\|_{\H}^{3/2}\|\hat \z \|_{\H}^{1/2} 
+ \|\z\|_{\H}^{1/2}\|\hat \z \|_{\H}^{3/2} + \|\hat \z \|_{\H}^{2}\big]
\leq  \frac{c}{|\lambda|^{1/6}}\big[\|\z\|_{\H}^2+\|\hat \z \|_{\H}^{2}\big],
$$
and \eqref{P3} follows.
Similarly, using again Lemma \ref{wW} and \eqref{wxxxxest2}, we find
\begin{align*}
|w_{xxx}(\ell_0)|^2 &\leq \|w_{xxx}\|_{L^\infty(\ell_0,\ell)}^2\\\noalign{\vskip0.5mm}
&\leq c \|w_{xxx}\|_{L^2(\ell_0,\ell)}^2 + c \|w_{xxx}\|_{L^2(\ell_0,\ell)}
\|w_{xxxx}\|_{L^2(\ell_0,\ell)}\\
&\leq c \|w_{xx}\|_{L^2(\ell_0,\ell)}^2+
c \|w_{xx}\|_{L^2(\ell_0,\ell)}\|w_{xxxx}\|_{L^2(\ell_0,\ell)} 
+ c \|w_{xx}\|_{L^2(\ell_0,\ell)}^{1/2}\|w_{xxxx}\|_{L^2(\ell_0,\ell)}^{3/2}\\
&\leq c|\lambda|^{5/6}\big[\|\z\|_{\H}^2+ \|\z\|_{\H}^{3/2}\|\hat \z \|_{\H}^{1/2}
+\|\z\|_{\H}\|\hat \z \|_{\H}
+ \|\z\|_{\H}^{1/2}\|\hat \z \|_{\H}^{3/2} + \|\hat \z \|_{\H}^{2}\big]\\
&\leq c|\lambda|^{5/6}\big[\|\z\|_{\H}^2+\|\hat \z \|_{\H}^{2}\big].
\end{align*}
The lemma has been proved.
\end{proof}

\begin{lemma}
\label{POINT2similar}
For every $|\lambda|$ sufficiently large the inequality
$$
|v_{xx}(0)| \leq \frac{c}{|\lambda|^{1/2}}|v_{xxx}(0)|
+ \frac{c}{|\lambda|^{3/4}}\big[\|\z\|_{\H} + \|\hat \z \|_{\H}\big]
$$
holds for some structural constant $c>0$ independent of $\lambda$.
\end{lemma}

\begin{proof}
The argument follows very closely the one of \cite[Lemma 3.10]{GRL} 
(see also \cite{LIULIU,SOZZO}).
Nevertheless, for the reader’s convenience, we report here the full proof.

Calling $\varphi_\lambda(x)=v_{xx}(x)-|\lambda|v(x)$ and combining
\eqref{R1s} with \eqref{R2s}, we easily find
$$\varphi_{\lambda xx} + |\lambda|\varphi_\lambda  = i \lambda \hat v + \hat V.$$
For $x \in (0,\ell_0)$ and $\lambda\neq0$, the solution of the equation above with the boundary conditions 
$\varphi_\lambda(0)=v_{xx}(0)$ and $\varphi_{\lambda x}(0)=v_{xxx}(0)$ can be written as
$$
\varphi_\lambda(x)= \frac{1}{2 |\lambda|^{1/2}} 
\big[J_{1\lambda}(x)+J_{2\lambda}(x)+iJ_{3\lambda}(x)+iJ_{4\lambda}(x)\big]
$$
having set
\begin{align*}
&J_{1\lambda}(x) = \big[|\lambda|^{1/2} v_{xx}(0)-i
v_{xxx}(0)\big] e^{i |\lambda|^{1/2}x},\\\noalign{\vskip2.5mm}
&J_{2\lambda}(x) = \big[|\lambda|^{1/2} v_{xx}(0)+ i
v_{xxx}(0)\big] e^{-i |\lambda|^{1/2}x},\\\noalign{\vskip1.5mm}
&J_{3\lambda}(x) =\int_0^x e^{i |\lambda|^{1/2}(s-x)} [i \lambda \hat v(s) + \hat V(s)] ds,\\
&J_{4\lambda}(x) =- \int_0^x e^{-i |\lambda|^{1/2}(s-x)} [i \lambda \hat v(s) + \hat V(s)] ds.
\end{align*}
Introducing the further variable $\psi_\lambda(x)=v_{x}(x)+|\lambda|^{1/2}v(x)$,
we infer that
$$\psi_{\lambda x}-|\lambda|^{1/2}\psi_\lambda=\varphi_\lambda.$$
Therefore, we obtain
\begin{align*}
\psi_\lambda(x) &= e^{|\lambda|^{1/2}(x-\ell_0)}\psi_\lambda(\ell_0) 
- \int_x^{\ell_{0}} e^{|\lambda|^{1/2}(x-s)}\varphi_\lambda(s) ds\\
&=e^{|\lambda|^{1/2}(x-\ell_0)}\psi_\lambda(\ell_0) 
-\int_x^{\ell_{0}} \frac{e^{|\lambda|^{1/2}(x-s)}}{2 |\lambda|^{1/2}} 
\big[J_{1\lambda}(s)+J_{2\lambda}(s)+iJ_{3\lambda}(s)+iJ_{4\lambda}(s)\big] ds.
\end{align*}
Choosing $x=0$ in the expression above and noting that $\psi_\lambda(0)=0$, we find
\begin{align*}
&\int_0^{\ell_{0}} e^{-|\lambda|^{1/2}s}
\big[J_{1\lambda}(s)+J_{2\lambda}(s)\big] ds\\
&= 2 |\lambda|^{1/2} e^{-|\lambda|^{1/2}\ell_0}\psi_\lambda(\ell_0)
-i\int_0^{\ell_{0}}e^{-|\lambda|^{1/2}s}
\big[J_{3\lambda}(s)+J_{4\lambda}(s)\big] ds.
\end{align*}
At this point, a straightforward computation yields the identity
$$
\int_0^{\ell_{0}} e^{-|\lambda|^{1/2}s}
\big[J_{1\lambda}(s)+J_{2\lambda}(s)\big] ds
= K_\lambda
+v_{xx}(0)
+\frac{v_{xxx}(0)}{|\lambda|^{1/2}},
$$
where
$$
K_\lambda=\frac{J_{1\lambda}(0) e^{(i-1)|\lambda|^{1/2}\ell_0}}{(i-1)|\lambda|^{1/2}}-
\frac{J_{2\lambda}(0) e^{-(i+1)|\lambda|^{1/2}\ell_0}}{(i+1)|\lambda|^{1/2}}.
$$
As a consequence, we have the bound
\begin{align*}
|v_{xx}(0)|&\leq \frac{1}{|\lambda|^{1/2}}|v_{xxx}(0)| + |K_\lambda|
+2 |\lambda|^{1/2} e^{-|\lambda|^{1/2}\ell_0}|\psi_\lambda(\ell_0)| \\
&\quad+\Big|\int_0^{\ell_{0}} e^{-|\lambda|^{1/2}s} 
\big[ J_{3\lambda}(s) + J_{4\lambda}(s) \big] ds\Big|.
\end{align*}
For every $|\lambda|$ sufficiently large, it is readily seen that
$$
|K_\lambda| \leq c e^{-|\lambda|^{1/2}\ell_0}
|v_{xx}(0)|+ \frac{c e^{-|\lambda|^{1/2}\ell_0} }{|\lambda|^{1/2}}|v_{xxx}(0)|
\leq  \frac12 |v_{xx}(0)|+ \frac{c}{|\lambda|^{1/2}}|v_{xxx}(0)|.$$
Moreover, invoking the Gagliardo-Nirenberg inequality, we also find
\begin{align*}
2 |\lambda|^{1/2} e^{-|\lambda|^{1/2}\ell_0}|\psi_\lambda(\ell_0)|
&\leq c |\lambda| e^{-|\lambda|^{1/2}\ell_0}\big[\|v_x\|_{L^\infty(0,\ell_0)}+
\|v\|_{L^\infty(0,\ell_0)}\big]\\\noalign{\vskip0.9mm}
&\leq c |\lambda| e^{-|\lambda|^{1/2}\ell_0}\|v_{xx}\|_{L^2(0,\ell_0)}\\\noalign{\vskip0.9mm}
&\leq c |\lambda| e^{-|\lambda|^{1/2}\ell_0}\|\z\|_\H\\
&\leq \frac{c}{|\lambda|^{3/4}}\|\z\|_\H,
\end{align*}
for every $|\lambda|$ sufficiently large.
This leads to 
\begin{align*}
|v_{xx}(0)|&\leq \frac{c}{|\lambda|^{1/2}}|v_{xxx}(0)| + \frac{c}{|\lambda|^{3/4}}\|\z\|_\H 
+\Big|\int_0^{\ell_{0}} e^{-|\lambda|^{1/2}s} \big[ J_{3\lambda}(s) + J_{4\lambda}(s) \big] ds\Big|.
\end{align*}
In order to finish the proof, we are left to show that 
\begin{equation}
\label{final}
\Big|\int_0^{\ell_{0}} e^{-|\lambda|^{1/2}s} \big[ J_{3\lambda}(s) + J_{4\lambda}(s) \big] ds\Big| \leq 
\frac{c}{|\lambda|^{3/4}}\|\hat \z\|_\H
\end{equation}
for every $|\lambda|$ sufficiently large. To this end, integrating by parts with 
$\hat v(0)=\hat v_{x}(0)=0$, we compute 
$$
\int_0^{\ell_{0}} e^{-|\lambda|^{1/2}s} J_{3\lambda}(s) ds
= P_{1\lambda} + P_{2\lambda}
+ P_{3\lambda},
$$
having set
\begin{align*}
&P_{1\lambda} = -\frac{e^{-(i+1)|\lambda|^{1/2}\ell_0}}{(i+1)|\lambda|^{1/2}}
\int_0^{\ell_0} e^{i |\lambda|^{1/2} s} [i \lambda \hat v(s) + \hat V(s)] ds,\\\noalign{\vskip2.5mm}
&P_{2\lambda} = -\frac{i\lambda}{(i+1)|\lambda|}\big[e^{-|\lambda|^{1/2} \ell_0} \hat v(\ell_0)+
\frac{e^{-|\lambda|^{1/2} \ell_0}}{|\lambda|^{1/2}} \hat v_x(\ell_0)\big],\\\noalign{\vskip1.5mm}
&P_{3\lambda} = \frac{i\lambda}{(i+1)|\lambda|^{3/2}}
\int_0^{\ell_0} e^{-|\lambda|^{1/2} s}\,\hat v_{xx}(s) ds + 
\frac{1}{(i+1)|\lambda|^{1/2}}
\int_0^{\ell_0} e^{-|\lambda|^{1/2} s}\,\hat V (s) ds.
\end{align*}
Using the Gagliardo-Nirenberg inequality and the H\"older inequality,
it is not difficult to check that for every $|\lambda|$ sufficiently large
\begin{align*}
&|P_{1\lambda}| \leq c|\lambda|^{1/2} e^{-|\lambda|^{1/2}\ell_0}\|\hat \z\|_\H\leq
\frac{c}{|\lambda|^{3/4}}\|\hat \z\|_\H,\\
&|P_{2\lambda}| \leq c e^{-|\lambda|^{1/2}\ell_0}\big[\|\hat v\|_{L^\infty(0,\ell_0)}+
\|\hat v_x\|_{L^\infty(0,\ell_0)}\big]
\leq \frac{c}{|\lambda|^{3/4}}\|\hat \z\|_\H,\\
&|P_{3\lambda}| \leq \frac{c}{|\lambda|^{3/4}}
\big[1- e^{-2|\lambda|^{1/2}\ell_0}\big]^{1/2}\big[\|\hat v_{xx}\|_{L^2(0,\ell_0)}+
\|\hat V\|_{L^2(0,\ell_0)}\big]
\leq \frac{c}{|\lambda|^{3/4}}\|\hat \z\|_\H.
\end{align*}
Therefore, we arrive at
$$
\Big|\int_0^{\ell_{0}} e^{-|\lambda|^{1/2}s} J_{3\lambda}(s) ds \Big| 
\leq 
\frac{c}{|\lambda|^{3/4}}\|\hat \z\|_\H.
$$
By the same token, we also have
$$
\Big|\int_0^{\ell_{0}} e^{-|\lambda|^{1/2}s} J_{4\lambda}(s) ds \Big| 
\leq 
\frac{c}{|\lambda|^{3/4}}\|\hat \z\|_\H,
$$
and \eqref{final} follows. 
\end{proof}

\begin{lemma}
\label{POINT2}
For every $|\lambda|\geq1$ the inequality
$$
|v_{xxx}(0)|^2 \leq c|\lambda|^{5/6}\big[\|\z\|_{\H}^2 + \|\hat \z \|_{\H}^2\big]
$$
holds for some structural constant $c>0$ independent of $\lambda$.
\end{lemma}

\begin{proof} 
Multiplying in $L^2(0,\ell_0)$ equation \eqref{R2s}
by $v_{xxx}$ and taking the real part, we obtain
$$
\Re \i \lambda \l V , v_{xxx} \r_{L^2(0,\ell_0)} + \Re \l v_{xxxx} , v_{xxx} \r_{L^2(0,\ell_0)} =
\Re \l \hat V , v_{xxx} \r_{L^2(0,\ell_0)}.
$$
After an elementary calculation, and with the aid of \eqref{R1s}, we find (the
horizontal bar stands for the complex conjugate)
$$
\Re \i \lambda \l V , v_{xxx} \r_{L^2(0,\ell_0)} 
= \Re \big[\i \lambda V(\ell_0) \overline{v_{xx}(\ell_0)}\big] +\frac12 |V_x(\ell_0)|^2 + 
\Re \l V_x, \hat v_{xx} \r_{L^2(0,\ell_0)}.
$$
In addition, it is readily seen that
$$
\Re \l v_{xxxx} , v_{xxx} \r_{L^2(0,\ell_0)}  
=\frac12|v_{xxx}(\ell_0)|^2-\frac12|v_{xxx}(0)|^2.
$$
Therefore, we get the bound
\begin{align*}
|v_{xxx}(0)|^2 &\leq c |v_{xxx}(\ell_0)|^2 + c|\lambda| |V(\ell_0)| |v_{xx}(\ell_0)| 
+ c |V_x(\ell_0)|^2 \\\noalign{\vskip0.5mm}
&\quad + c \|V_x\|_{L^2(0,\ell_0)}\|\hat \z\|_{\H}+\|v_{xxx}\|_{L^2(0,\ell_0)}\|\hat \z\|_{\H}. 
\end{align*}
We now estimate for $|\lambda|\geq1$ the terms appearing in the right-hand side above.
First we observe that, since 
$v_{xxx}(\ell_0)=w_{xxx}(\ell_0)-W_{x}(\ell_0)$, an exploitation \eqref{P2} and \eqref{P4} entails
$$
|v_{xxx}(\ell_0)|^2\leq c|\lambda|^{5/6}\big[\|\z\|_{\H}^2 + \|\hat \z \|_{\H}^2\big].
$$
Recalling also that $V(\ell_0)=W(\ell_0)$ and $v_{xx}(\ell_0)=w_{xx}(\ell_0)$, using \eqref{P1} and \eqref{P3}
we get
\begin{align*}
|\lambda| |V(\ell_0)| |v_{xx}(\ell_0)|  &\leq
c |\lambda|^{3/4} \big[\|\z\|_{\H}^2 +\|\z\|_{\H}\|\hat \z \|_{\H}+ \|\hat \z \|_{\H}^2\big]\\
&\leq c |\lambda|^{3/4} \big[\|\z\|_{\H}^2 +\|\hat \z \|_{\H}^2\big].
\end{align*}
Moreover, since $V_x(\ell_0)=W_x(\ell_0)$, in the light of \eqref{P2}
we have
$$
|V_x(\ell_0)|^2 \leq c|\lambda|^{2/3}\big[\|\z\|_{\H}^2 + \|\hat \z \|_{\H}^2\big].
$$
Next, interpolating between $L^2(0,\ell_0)$ and
$H^2(0,\ell_0)$ and owing to \eqref{R1s}, we obtain
\begin{align*}
\|V_x\|_{L^2(0,\ell_0)}\|\hat \z\|_{\H}  
&\leq c\|V\|_{L^2(0,\ell_0)}^{1/2} \|V_{xx}\|_{L^2(0,\ell_0)}^{1/2}\|\hat \z\|_{\H}\\
&\leq  c |\lambda|^{1/2}\big[\|\z\|_{\H}\|\hat \z\|_{\H} +\|\z\|_{\H}^{1/2}\|\hat \z\|_{\H}^{3/2}\big]
\\ 
&\leq c |\lambda|^{1/2}\big[\|\z\|_{\H}^2 +\|\hat \z\|_{\H}^2\big].
\end{align*}
Similarly, interpolating between $H^2(0,\ell_0)$ and
$H^4(0,\ell_0)$ and using \eqref{R2s}, we find
\begin{align*}
\|v_{xxx}\|_{L^2(0,\ell_0)} \|\hat \z\|_{\H} &\leq c\|v_{xx}\|_{L^2(0,\ell_0)}
\|\hat \z\|_{\H}+
c\|v_{xx}\|_{L^2(0,\ell_0)}^{1/2} \|v_{xxxx}\|_{L^2(0,\ell_0)}^{1/2} 
\|\hat \z\|_{\H}\\
&\leq  c |\lambda|^{1/2}\big[\|\z\|_{\H}\|\hat \z\|_{\H} +\|\z\|_{\H}^{1/2}\|\hat \z\|_{\H}^{3/2}\big]
\\ 
&\leq c |\lambda|^{1/2}\big[\|\z\|_{\H}^2 +\|\hat \z\|_{\H}^2\big].
\end{align*}
Collecting all the estimates above, and noting 
that the biggest power is $|\lambda|^{5/6}$, we arrive at the desired conclusion.
\end{proof}

\begin{lemma}
\label{vV}
For every $|\lambda|$ sufficiently large the inequality
$$
\|v_{xx}\|_{L^2(0,\ell_0)}^2 + \|V\|_{L^2(0,\ell_0)}^2 
\leq \frac{c}{|\lambda|^{1/12}}\big[\| \z\|_\H^2 + \|\hat  \z\|_\H^2 \big]
$$
holds for some structural constant $c>0$ independent of $\lambda$.
\end{lemma}

\begin{proof}
Setting $q(x)=x-\ell_0$ for $x\in(0,\ell_0)$, we
multiply in $L^2(0,\ell_0)$ equation \eqref{R2s}
by $q v_{x}$. Taking the real part of the resulting identity, we obtain
$$
\Re \i \lambda \l V , q v_x \r_{L^2(0,\ell_0)} + \Re \l v_{xxxx} , q v_x \r_{L^2(0,\ell_0)} =
\Re \l \hat V , q v_x \r_{L^2(0,\ell_0)}.
$$
Exploiting \eqref{R1s} and after an elementary calculation, we find
$$
\Re \i \lambda \l V , q v_x \r_{L^2(0,\ell_0)} 
= \frac12\|V\|^2_{L^2(0,\ell_0)} - \Re \l V , q\hat v_x \r_{L^2(0,\ell_0)}. 
$$
Moreover, we compute (as before the horizontal bar stands for the complex conjugate)
\begin{align*}
\Re \l v_{xxxx} , q v_x \r_{L^2(0,\ell_0)}  
= -\Re \big[v_{xx}(\ell_0) \overline{v_{x}(\ell_0)}\big]
+\frac32\|v_{xx}\|_{L^2(0,\ell_0)}^2- \frac{\ell_0}{2}|v_{xx}(0)|^2.
\end{align*}
As a consequence, we find
\begin{align*}
&\frac32\|v_{xx}\|_{L^2(0,\ell_0)}^2 + \frac12\big\|V\|^2_{L^2(0,\ell_0)}\\\nonumber
&= \frac{\ell_0}{2}|v_{xx}(0)|^2 +\Re \big[v_{xx}(\ell_0) \overline{v_{x}(\ell_0)}\big]
+ \Re \l \hat V, q v_x \r_{L^2(0,\ell_0)}+\Re \l V , q\hat v_x \r_{L^2(0,\ell_0)}.
\end{align*}
Collecting Lemmas \ref{POINT2similar} and \ref{POINT2}, it is readily seen that
for every $|\lambda|$ sufficiently large 
$$
\frac{\ell_0}{2}|v_{xx}(0)|^2 \leq 
\frac{c}{|\lambda|}|v_{xxx}(0)|^2
+ \frac{c}{|\lambda|^{3/2}}\big[\|\z\|_{\H}^2 + \|\hat \z \|_{\H}^2\big]\leq
\frac{c}{|\lambda|^{1/6}}\big[\|\z\|_{\H}^2 + \|\hat \z \|_{\H}^2\big].
$$
Recalling that 
$v_{xx}(\ell_0)=w_{xx}(\ell_0)$, and exploiting 
the Gagliardo-Nirenberg inequality together with \eqref{P3}, we also obtain
\begin{align*}
|\Re \big[v_{xx}(\ell_0) \overline{v_{x}(\ell_0)}\big]| 
& \leq \frac{c}{|\lambda|^{1/12}} \big[\| \z\|_\H + \|\hat  \z\|_\H \big]\|v_x\|_{L^\infty(0,\ell_0)}\\
& \leq \frac{c}{|\lambda|^{1/12}} \big[\| \z\|_\H + \|\hat  \z\|_\H \big]\|v_{xx}\|_{L^2(0,\ell_0)}\\
& \leq \frac{c}{|\lambda|^{1/12}} \big[\| \z\|_\H^2 + \|\hat  \z\|_\H^2 \big].
\end{align*}
Therefore, we arrive at the bound
\begin{align}
\label{quasifine}
&\|v_{xx}\|_{L^2(0,\ell_0)}^2 + \|V\|_{L^2(0,\ell_0)}^2\\\nonumber
&\leq \frac{c}{|\lambda|^{1/12}}\big[\| \z\|_\H^2 + \|\hat  \z\|_\H^2 \big]
+ |\Re \l \hat V, q v_x \r_{L^2(0,\ell_0)}| + |\Re \l V, q \hat v_x \r_{L^2(0,\ell_0)}|.
\end{align}
Next, in the light of \eqref{R1s}, we estimate 
\begin{align*}
|\Re \l \hat V, q v_x \r_{L^2(0,\ell_0)}|
&\leq \frac{c}{|\lambda|}\big[\|V_x\|_{L^2(0,\ell_0)}+\|\hat v_{xx}\|_{L^2(0,\ell_0)}\big]\|\hat V\|_{L^2(0,\ell_0)}\\
&\leq  \frac{c}{|\lambda|} \|V_x\|_{L^2(0,\ell_0)}\|\hat \z \|_{\H} + \frac{c}{|\lambda|} \|\hat \z \|_{\H}^2.
\end{align*}
Interpolating between $L^2(0,\ell_0)$ and
$H^2(0,\ell_0)$ and invoking again \eqref{R1s}, we also have
$$
\|V_x\|_{L^2(0,\ell_0)}\leq c \|V\|_{L^2(0,\ell_0)}^{1/2}\|V_{xx}\|_{L^2(0,\ell_0)}^{1/2}
\leq c|\lambda|^{1/2}\|\z\|_{\H} + c \|\z\|_{\H}^{1/2}\|\hat \z\|_{\H}^{1/2}.
$$
In conclusion, we find
\begin{align}
\label{R1xxx}
|\Re \l \hat V, q v_x \r_{L^2(0,\ell_0)}|&\leq \frac{c}{|\lambda|^{1/2}}
\big[\|\z\|_{\H}\|\hat \z\|_{\H}+\|\z\|_{\H}^{1/2}\|\hat \z\|_{\H}^{3/2}+\|\hat \z\|_{\H}^2\big]\\\nonumber
&\leq \frac{c}{|\lambda|^{1/2}}
\big[\|\z\|_{\H}^2 + \|\hat \z\|_{\H}^2\big].
\end{align}
At this point, owing to  \eqref{R2s}, we compute
\begin{align*}
|\Re \l V, q \hat v_x \r_{L^2(0,\ell_0)}|
&\leq \frac{1}{|\lambda|}|\Re \l v_{xxxx}, q \hat v_x \r_{L^2(0,\ell_0)}\big|
+ \frac{c}{|\lambda|} \|\hat V \|_{L^2(0,\ell_0)}\|\hat v_{xx} \|_{L^2(0,\ell_0)}\\
&= \frac{1}{|\lambda|}|\Re \l v_{xxx}, \hat v_x + q\hat v_{xx} \r_{L^2(0,\ell_0)}|
+ \frac{c}{|\lambda|} \|\hat V \|_{L^2(0,\ell_0)}\|\hat v_{xx} \|_{L^2(0,\ell_0)}\\
&\leq \frac{c}{|\lambda|} \|v_{xxx}\|_{L^2(0,\ell_0)} \| \hat \z \|_\H+ \frac{c}{|\lambda|} \|\hat \z \|_{\H}^2.
\end{align*}
Interpolating between $H^2(0,\ell_0)$ and
$H^4(0,\ell_0)$ and exploiting again \eqref{R2s},  we have
\begin{align*}
\|v_{xxx}\|_{L^2(0,\ell_0)}&\leq 
c\|v_{xx}\|_{L^2(0,\ell_0)} + c\|v_{xx}\|_{L^2(0,\ell_0)}^{1/2}\|v_{xxxx}\|_{L^2(0,\ell_0)}^{1/2}\\
&\leq c |\lambda|^{1/2}\big[\|\z\|_{\H} + \|\z\|_{\H}^{1/2}\|\hat \z\|_{\H}^{1/2}\big]\\
&\leq c |\lambda|^{1/2}\big[\|\z\|_{\H} + \|\hat \z\|_{\H}\big].
\end{align*}
Therefore, we arrive at
$$
|\Re \l V, q \hat v_x \r_{L^2(0,\ell_0)}|
\leq \frac{c}{|\lambda|^{1/2}} \big[\|\z\|_{\H} \|\hat \z\|_{\H}+\|\hat \z\|_{\H}^2 \big]
\leq \frac{c}{|\lambda|^{1/2}} \big[\|\z\|_\H^2+\|\hat \z\|_\H^2 \big].
$$
Plugging \eqref{R1xxx} and the estimate above into \eqref{quasifine} we reach the thesis.
\end{proof}

We are now in a position to complete the proof of Theorem \ref{MAINTH}. Indeed, collecting Lemmas \ref{wW} and
\ref{vV}, we infer that
$$
\|\z\|_\H^2 \leq \frac{c}{|\lambda|^{1/12}}\big[\| \z\|_\H^2 + \|\hat  \z\|_\H^2 \big]\leq
\frac12 \| \z\|_\H^2+  \frac{c}{|\lambda|^{1/12}} \|\hat  \z\|_\H^2,
$$
for every $|\lambda|$ sufficiently large. Therefore
$$
\|\z\|_\H\leq \frac{c}{|\lambda|^{1/24}} \|\hat  \z\|_\H,
$$
which in turns yields (recall that $\sigma(\A) \cap \i \R = \emptyset$) 
$$
\limsup_{|\lambda|\to\infty} |\lambda|^{1/24}\|(i\lambda - \A)^{-1}\|_{L(\H)}<\infty.
$$
The latter is nothing but \eqref{limsup} with $\mu=1/24$, as desired. 
\qed

\begin{Acknowledgments}
The research of M.\ Caggio leading to these results has received funding from 
the Czech Sciences Foundation (GA\v CR), 22-01591S.  Moreover,  M.\ Caggio   
has been supported by  Praemium Academiae of \v S.\ Ne\v casov\' a. CAS 
is supported by RVO:67985840. F. Dell'Oro  is member of 
Gruppo Nazionale per l'Analisi Matematica, la Probabilit\`a e le loro Applicazioni (GNAMPA), 
Istituto Nazionale di Alta Matematica.
\end{Acknowledgments}




\begin{thebibliography}{99}

\bibitem{ADAMS}
{\au R.A. Adams and J. Fournier},
{\bk Sobolev spaces. Second edition},
\eds{Elsevier/Academic Press}{Amsterdam}{2003. xiv+305 pp}

\bibitem{AHR}
{\au K. Ammari, F. Hassine and L. Robbiano},
{\ti Stabilization for vibrating plate with singular structural damping},
{\jou Discrete Contin.\ Dyn.\ Syst.\ Ser.\ S} (in press).

\bibitem{AB}
{\au W. Arendt and C.J.K. Batty},
{\ti Tauberian theorems and stability of one-parameter semigroups},
{\jou Trans.\ Amer.\ Math.\ Soc.}
\no{306}{837--852}{1988}

\bibitem{BARRAZA}
{\au B. Barraza Martinez, R. Denk, J. Hern\'andez Monz\'on, F. Kammerlander and M. Nendel},
{\ti Regularity and asymptotic behavior for a damped plate-membrane transmission problem},
{\jou J.\ Math.\ Anal.\ Appl.}
\no{474}{1082--1103}{2019}

\bibitem{BELA}
{\au B. Belinskiy and I. Lasiecka},
{\ti Gevrey's and trace regularity of a semigroup associated 
with beam equation and non-monotone boundary conditions},
{\jou J.\ Math.\ Anal.\ Appl.}
\no{332}{137--154}{2007}

\bibitem{BURQ}
{\au N. Burq},
{\ti D\'ecroissance de l'\'energie locale de l'\'equation des ondes pour le probl\`eme 
ext\'erieur et absence de r\'esonance au voisinage du r\'eel},
{\jou Acta Math.}
\no{180}{1--29}{1998}

\bibitem{CT}
{\au S. Chen and R. Triggiani},
{\ti Gevrey class semigroups arising from elastic systems 
with gentle dissipation:\ the case $0<\alpha<\tfrac12$},
{\jou Proc.\ Amer.\ Math.\ Soc.}
\no{110}{401--415}{1990}

\bibitem{DENK}
{\au R. Denk and F. Kammerlander}, 
{\ti Exponential stability for a coupled system of damped-undamped plate equations},
{\jou IMA J.\ Appl.\ Math.}
\no{83}{302--322}{2018}

\bibitem{ENG}
{\au K.-J. Engel, R. Nagel},
{\bk One-parameter semigroups for linear evolution equations},
\eds{Springer-Verlag}{New York}{2000}

\bibitem{GER}
{\au L. Gearhart},
{\ti Spectral theory for contraction semigroups on Hilbert space},
{\jou Trans.\ Amer.\ Math.\ Soc.}
\no{236}{385--394}{1978}

\bibitem{GRL}
{\au G. G\'omez \'Avalos, J.E. Mu{\~n}oz Rivera and Z. Liu} ,
{\ti Gevrey class of locally dissipative Euler-Bernoulli beam equation},
{\jou SIAM J. Control Optim.}
\no{59}{2174--2194}{2021}

\bibitem{Kat}
{\au T. Kato},
{\bk Perturbation theory for linear operators},
\eds{Springer-Verlag}{New York}{1980}

\bibitem{KOM}
{\au V. Komornik},
{\bk Exact controllability and stabilization. The multiplier method},
\eds{Masson}{Paris}{1994}

\bibitem{LAG}
{\au J. Lagnese},
{\bk Boundary stabilization of thin plates},
\eds{SIAM}{Philadelphia}{1989}

\bibitem{LAS1}
{\au I. Lasiecka},
{\ti Exponential decay rates for the solutions of Euler-Bernoulli 
equations with boundary dissipation occurring in the moments only},
{\jou  J.\ Differential Equations}
\no{95}{169--182}{1992}

\bibitem{LAS}
{\au I. Lasiecka},
{\bk Mathematical control theory of coupled PDEs},
\eds{Society for Industrial and Applied Mathematics (SIAM)}{Philadelphia, PA}{2002}

\bibitem{LIULIU}
{\au K. Liu and Z. Liu}, 
{\ti Exponential decay of energy of the Euler-Bernoulli 
beam with locally distributed Kelvin-Voigt damping},
{\jou SIAM J. Control. Optim.}
\no{36}{1086--1098}{1998}

\bibitem{PAZ}
{\au A. Pazy},
{\bk Semigroups of linear operators and applications to partial differential equations},
\eds{Springer-Verlag}{New York}{1983}

\bibitem{PRU}
{\au J. Pr\"{u}ss},
{\ti On the spectrum of $\text{C}_0$-semigroups},
{\jou Trans.\ Amer.\ Math.\ Soc.}
\no{284}{847--857}{1984}

\bibitem{SOZZO}
{\au B. Sozzo and J.E. Mu{\~n}oz Rivera},
{\ti The Gevrey class of the Euler-Bernoulli beam model}, 
{\jou J.\ Math.\ Anal.\ Appl.} 
\no{505}{Paper No.\ 125619, 13 pp}{2022}

\bibitem{TAYLOR}
{\au S. Taylor}
{\bk Ph.D.\ Thesis, Chapter ``Gevrey semigroups"},
School of Mathematics, University of Minnesota, 1989.

\bibitem{TEB}
{\au L. Tebou},
{\ti Well-posedness and stability of a hinged plate equation 
with a localized nonlinear structural damping},
{\jou Nonlinear Anal.} 
\no{71}{e2288--e2297}{2009}

\end{thebibliography}
\end{document}